\documentclass[preprint,11pt]{elsarticle}
\usepackage{mathrsfs}
\usepackage{amsfonts}
\usepackage{amsmath}
\usepackage{amstext}
\usepackage{stmaryrd}
\usepackage{amssymb}
\usepackage{amsthm}
\usepackage{mathrsfs}
\usepackage{amsfonts}
\usepackage{enumerate}
\usepackage{amscd}

\usepackage{indentfirst}
\usepackage{enumerate}
\usepackage{amsmath,amsfonts,amssymb,amsthm}
\usepackage{amsmath,amssymb,amsthm,amscd}
\usepackage{graphicx,mathrsfs}
\usepackage{appendix}
\usepackage{lineno}

\numberwithin{equation}{section}

\newtheorem{theorem}{Theorem}[section]
\newtheorem{proposition}[theorem]{Proposition}
\newtheorem{corollary}[theorem]{Corollary}
\newtheorem{lemma}[theorem]{Lemma}

\theoremstyle{definition}

\theoremstyle{remark}
\newtheorem{remark}[theorem]{Remark}
\newtheorem{remarks}[theorem]{Remarks}

\begin{document}
\begin{frontmatter}

\title{On radial symmetry of rotating vortex patches in the disc
%\tnoteref{sup}} \tnotetext[sup]{This work is partially supported by the NSFC under the grants 11371282
}

\author[1]{Guodong Wang}
\ead{wangguodong14@mails.ucas.ac.cn}
\author[2]{Bijun Zuo}
\ead{bjzuo@amss.ac.cn}

\address[1]{Institute for Advanced Study in Mathematics, Harbin Institute of Technology, Harbin {\rm150001}, P. R. China}
\address[2]{Institute of Applied Mathematics, Chinese Academy of Science, Beijing 100190, and University of Chinese Academy of Sciences, Beijing 100049,  P.R. China}
%\address[2]{Institute of Applied Mathematics, Chinese Academy of Sciences, Beijing {\rm100190}, China}
%\address[2]{University of Chinese Academy of Sciences, Beijing {\rm100049},  China}

%% use optional labels to link authors explicitly to addresses:
%% \author[label1,label2]{}
%% \address[label1]{}
%% \address[label2]{}

\begin{abstract}
In this note, we consider the radial symmetry property of rotating vortex patches for the 2D incompressible Euler equations in the unit disc. By choosing a suitable vector field to deform the patch, we show that each simply-connected rotating vortex patch $D$ with angular velocity $\Omega$, $\Omega\geq \max\{{1}/{2},({2 l^2})/{(1-l^2)^2}\}$ or $\Omega\leq -({2 l^2})/{(1-l^2)^2}$, where $l=\sup_{x\in D}|x|$, must be a disc. The main idea of the proof, which has a variational flavor, comes from a very recent paper of G\'omez-Serrano--Park--Shi--Yao, arXiv:1908.01722, where radial symmetry of rotating vortex patches in the whole plane was studied.
\end{abstract}
\end{frontmatter}

%\addcontentsline{toc}{section}{References}
%\newpage

%\tableofcontents
%\section{Introduction}
%\addcontentsline{toc}{section}{Introduction}

\section{Introduction}
Let $\mathbb{D}_r$ be the disc centered at the origin with radius $r$, that is, \[\mathbb{D}_r:=\{x\in\mathbb R^2\mid x=(x_1,x_2), |x|:=x_1^2+x_2^2<r\}.\]
In this note, we shall study the radial symmetry property of solutions to the two-dimensional incompressible Euler equations in $\mathbb{D}_r$
\begin{equation}\label{euler}
\begin{cases}
  \partial_t\mathbf{v}+(\mathbf{v}\cdot\nabla)\mathbf{v}=-\nabla P, & (x,t)\in\mathbb{D}_r\times\mathbb R_+, \\
  \nabla\cdot\mathbf{v}=0,\\
  \mathbf{v}\cdot\mathbf{n}=0, & x\in\partial\mathbb{D}_r, \\
  \mathbf{v}|_{t=0}=\mathbf{v}_0,
\end{cases}
\end{equation}
where $\mathbf{v}=(v_1,v_2)$ is the velocity field, $P$ is the scalar pressure and $\mathbf{n}$ is the outward unit normal of $\partial \mathbb {D}_r$. The boundary condition $\mathbf{v}\cdot\mathbf{n}=0$, which is usually called the impermeability boundary condition, means that there is no matter flow through $\partial \mathbb{D}_r$. By introducing the scalar vorticity $\omega:=(\partial_{x_1}v_2-\partial_{x_2}v_1)$, the Euler system can be simplified as a single equation for the vorticity
\begin{equation}\label{vor}
\begin{cases}
\partial_t\omega+\nabla^\perp(\mathcal{G}\omega)\cdot\nabla\omega=0,&(x,t)\in\mathbb{D}_{r}\times \mathbb R_+,\\
\omega|_{t=0}=\omega_0,&x\in\mathbb{D}_r,
\end{cases}
\end{equation}
where $\nabla^\perp:=(\partial_{x_2},-\partial_{x_1})$ and $\mathcal{G}_r\omega(x):=\int_{\mathbb{D}_r}G_r(x,y)\omega(y)dy,$ with $G_r$ being the Green function for $-\Delta$ in $\mathbb{D}_r$ with zero Dirichlet data
\[G_r(x,y)=-\frac{1}{2\pi}\ln|x-y|-h_r(x,y),\,\,\,h_r(x,y)=-\frac{1}{2\pi}\ln\left|\frac{rx}{|x|}-\frac{|x|y}{r}\right|.\]

Equation \eqref{vor} is usually called the vorticity equation. For weak solutions of the vorticity equation with initial vorticity $\omega_0\in L^\infty(\mathbb{D}_r)$, the global well-posedness result was proved by Yudovich \cite{Y}. See also \cite{MB}, Chapter 8. Since $\nabla^\perp(\mathcal{G}\omega)$ is a divergence-free vector field, it is easy to see that the distribution function of the solution of \eqref{vor} is independent of the time variable $t$, that is,
\[|\{x\in\mathbb D_r\mid \omega(x,t)>a \}|=|\{x\in\mathbb D_r\mid \omega_0(x)>a \}|\]
for all $a\in\mathbb R$ and $t\in \mathbb R_+$. Here and in the sequel, we use $|\cdot|$ to denote the two-dimensional Lebesgue measure. As a consequence, if the initial vorticity $\omega_0$ is a constant multiple of the characteristic function of some measurable set $D\subset \mathbb D_r$, that is, $\omega_0=\lambda I_{D}$, where $\lambda\in\mathbb R$ represents the vorticity strength, then the evolved vorticity $\omega(\cdot,t)$ must be of the form
$\omega(\cdot,t)=\lambda I_{D_t}$ with $|D_t|=|D|$ for all $t>0$. The preservation of regularity for the boundary of an evolving vortex patch was firstly proved by Chemin \cite{Chemin} for the whole plane case and then was extended to bounded domains by Depauw \cite{Depauw}.

In this paper, we are mainly concerned with rotating solution, also called $V$-state, of the Euler equations, that is, solution with the form
\begin{equation}\label{1-2}
  \omega(x,t)=w(e^{-i\Omega t}x),
\end{equation}
where $e^{-i\Omega t}x$ denotes clockwise rotation through $\pi/2$ of $x$, and $\Omega\in\mathbb R$ is the angular velocity of the rotating solution. It is easy to see that the solution rotates clockwisely if $\Omega<0$, and rotates anticlockwisely if $\Omega>0$. If $\Omega=0$, then obviously $w$ is a steady solution.

For smooth $w$, we can substitute \eqref{1-2} into the vorticity equation \eqref{vor} to obtain
\begin{equation}\label{1-3}
 \nabla\cdot(w(x)\nabla^\perp(\mathcal G_r w(x)+\frac{\Omega}{2}|x|^2))=0.
\end{equation}
Since we are going to deal with solutions with discontinuity, we need to interpret \eqref{1-3} in the following weak sense
\begin{equation}\label{1-4}
 \int_{\mathbb D_r}w(x)\nabla^\perp(\mathcal G_r w(x)+\frac{\Omega}{2}|x|^2)\cdot \nabla\phi(x)dx=0,\,\,\forall\,\phi\in C_c^\infty(\mathbb D_r).
\end{equation}
In fact, \eqref{1-3} can be obtained by integration by parts(notice that $\nabla^\perp(\mathcal G_r w(x)+\frac{\Omega}{2}|x|^2))$ is divergence-free).
We call $\omega$ a rotating vortex patch, or just rotating patch for brevity, if it satisfies \eqref{1-4} and has the form
\begin{equation}\label{form}
\omega(x,t)=w(e^{-i\Omega t}x), \,\,w=\lambda I_{D},
\end{equation}
where $\lambda$ is a parameter representing the vorticity strength of the patch.
In the sequel we shall also call the set $D$ in \eqref{form} a rotating patch.
It is not difficult to check that if $D$ is rotating patch with $C^1$ boundary and $\lambda\neq 0$, then $\lambda \mathcal G_r I_D+\frac{\Omega}{2}|x|^2$ is a constant on each connected component of $\partial D$(although on different components the constants may be different). In fact, since $D$ is a rotating patch, we have
 \begin{equation}\label{1-101}
 \lambda\int_{D}\nabla^\perp(\lambda\mathcal G_r I_D(x)+\frac{\Omega}{2}|x|^2))\cdot \nabla\phi(x)dx=0,\,\,\forall\,\phi\in C_c^\infty(\mathbb D_r).
\end{equation}
Then by integration by parts, we obtain
\begin{equation}\label{1-102} \int_{\partial D}\phi(x)\nabla^\perp(\lambda\mathcal G_r I_D(x)+\frac{\Omega}{2}|x|^2))\cdot\vec{\nu}(x)d\mathcal H^1=0,
\end{equation}
where $\vec{\nu}$ denotes the outward unit normal of $\partial D$ and $d\mathcal H^1$ denotes the one-dimensional Hausdorff measure. Taking into account the fact that $\phi$ can be chosen arbitrarily in $C_c^\infty(\mathbb D_r)$, we get
\[\nabla^\perp(\lambda\mathcal G_r I_D(x)+\frac{\Omega}{2}|x|^2))\cdot\vec{\nu}(x)\equiv0,\,\, x\in\partial D,\]
 which leads to the desired result.

In the literature, there are a large number of results on existence and stability of rotating patches(including stationary patches) in the whole plane and in the disc. For existence, roughly speaking, there are mainly two types of rotating patches in the literature. The first one is of desingularization type. As the name suggests, it is about the desingularization of point vortices. More precisely, desingularization of point vortices is to construct a family of rotating(steady) vortex patches of the Euler equations that ``shrinks" to a given rotating(steady) system of point vortices. We point the interested reader to \cite{CWW}\cite{CWWZ}\cite{W} and the references listed therein. The second type of rotating patches is of bifurcation type. It consists of finding a new rotating patch bifurcating from a given one(for example, a disc, an annulus or a Kirchhoff ellipse). Related references are \cite{Burbea}\cite{dHHM}\cite{dHMV}\cite{DZ}\cite{HMV}. There are also many efforts that have been devoted to establishing the stability or instability of rotating patches. See \cite{GHS}\cite{Ta}\cite{W} for example. It should be noted that steady patches in general bounded domains have also been studied by many authors in recent years. See \cite{CPY}\cite{CW}\cite{CW3}\cite{T} for example.

Now we come back to rotating patches in the disc $\mathbb D_r$. It is easy to see that if $D$ is a disc centered at the origin, then it must be a rotating patch with arbitrary angular velocity. Now a very natural problem arises: \textit{under what conditions on $\lambda, r, \Omega$ must a rotating patch $D$ be a disc centered at the origin?}
To answer this question, we first notice the following fact which can be easily checked by using the scaling property of the vorticity equation \eqref{vor}: $D$ is a rotating patch in $\mathbb D_r$ with vorticity strength $\lambda$ and angular velocity $\Omega$ if and only if $D/r:=\{x\in\mathbb R^2\mid rx\in D\}$ is a rotating patch in $\mathbb D_1$ with vorticity strength $1$ and angular velocity $\Omega/\lambda$. For this reason, we will only consider the case $\lambda=1$ and $r=1$ in the rest of this paper.

Before we state our main result, we shall briefly review some known results on the radial symmetry property of rotating patches, both in the whole plane and in the unit disc. For rotating patches in the whole plane, Hmidi \cite{Hm} proved that any $C^1$ simply-connected rotating patch with angular velocity $\Omega$ must be radial if $\Omega=1/2$, or $\Omega<0$ but with some extra convexity assumption. For each $\Omega\in(0,1/2)$, de la Hoz--Hmidi--Mateu--Verdera \cite{dHMV} proved existence of non-radial rotating patches with $m$-fold symmetry bifurcating at $\Omega$. Recently, G\'omez-Serrano--Park--Shi--Yao \cite{GPSY} completely solved the radial symmetry problem for rotating patches by showing that any $C^1$ rotating patch(can be non-simply-connected) with angular velocity $\Omega\in(-\infty,0)\cup[1/2,+\infty)$ must be radially symmetric, and if $\Omega=0,$ then this rotating patch must be radially symmetric up to a translation.
As for rotating patches in the unit disc, to our knowledge, there is no result on radial symmetry in the literature by now. Here we only recall two existence results. In \cite{dHHM}, based on bifurcation theory, de la Hoz--Hassainia--Hmidi--Mateu proved that for any $b\in(0,1)$ and $m$ a positive integer, there exists a family of $m$-fold symmetric rotating patches bifurcating from the steady patch $\mathbb D_b$, $b\in(0,1)$, with angular velocity $\Omega_m=(m-1+b^{2m})/(2m)$. These rotating patches are simply-connected, moreover, the angular velocity lies in the interval $(0,1/2)$ just as the whole plane case. In \cite{CWWZ}, Cao-Wan-Wang-Zhan studied the existence of rotating patches of desingularization type in the unit disc. They proved that for any fixed $\Omega>0$, there exists a family of $C^1$ simply-connected rotating patches $\omega^\lambda=\lambda I_{D^\lambda}(e^{-i\Omega t}x)$ with $\lambda$ sufficiently large, moreover, $D^\lambda$ is supported in a very small region near some point $x_\Omega$ with $|x_\Omega|=0$ if $\Omega\leq (2\pi)^{-1}$ and $|x_\Omega|=(1-(2\pi\Omega)^{-1})^{1/2}$ if $\Omega> (2\pi)^{-1}$. It is easy to see that in this situation the angular velocity still lies in the interval $(0,1/2)$.

One may ask whether all the non-radial rotating patches in the unit disc possess an angular velocity in the interval $(0,1/2)$. This is in general a difficult problem. In this paper, we partially solve this problem by showing that for any $C^1$ simply-connected rotating patch with vorticity strength $1$ and angular velocity $\Omega$, if $\Omega\geq \max\{{1}/{2},({2 l^2})/{(1-l^2)^2}\}$ or $\Omega\leq -({2 l^2})/{(1-l^2)^2}$, where $l=\sup_{x\in D}|x|$, then $D$ must be radial.

The proof is inspired by a recent paper \cite{GPSY} by G\'omez-Serrano--Park--Shi--Yao. The basic idea is as follows.
Define
\begin{equation}\label{E}
  E_\Omega(D):=\frac{1}{2}\int_D\int_DG_1(x,y)dxdy+\frac{\Omega}{2}\int_D|x|^2dx,
\end{equation}
where $G_1$ is the Green function in $\mathbb D_1$.
Since $D$ is a rotating patch, we can easily check that $E_\Omega$
is a critical point on the following rearrangement class
\[\mathcal R(D):=\{K\subset \mathbb D_1\mid |K|=|D|\}\]
in the sense that if we deform $D$ without changing its area, the variation of $E_\Omega$ is zero. On the other hand, we can choose a suitable divergence-free field $\mathbf{w}$ to deform $D$ and calculate the variation of $E_\Omega$. We will show that if $D$ is not a disc centered at the origin, then the variation of $E_\Omega$ is not zero for sufficiently large $|\Omega|$.

This paper is organized as follows. In Section 2 we state the main result and give several remarks. In Section 3 we prove the main result.

\section{Main Result}
In the rest of this paper, we will only focus on rotating patches in $\mathbb D_1$ with unit vorticity strength. For brevity, we will  use $G$(rather than $G_1$) to denote the Green function in $\mathbb D_1$, that is,
\[G(x,y)=-\frac{1}{2\pi}\ln|x-y|-h(x,y),\,\,h(x,y)=-\frac{1}{2\pi}\ln\left|\frac{x}{|x|}-|x|y\right|.\]
The corresponding Green operator $\mathcal{G}$ is defined by
\[\mathcal{G}\omega(x):=\int_DG(x,y)\omega(y)dy.\]

Our main theorem can be stated as follows.

\begin{theorem}\label{thm}
Let $D$ be a $C^1$ simply-connected rotating vortex patch in $\mathbb D_1$ with vorticity strength 1 and angular velocity $\Omega$, that is, $I_{D}(e^{-i\Omega t}x)$ is a solution to the vorticity equation \eqref{vor} with $r=1$. If $\Omega\geq\max\{\frac{1}{2},\frac{2 l^2}{(1-l^2)^2}\}$ or $\Omega\leq-\frac{2 l^2}{(1-l^2)^2}$, where $l=\sup_{x\in D}|x|$, then $D$ must be a disc centered at the origin.

\end{theorem}

\begin{remarks}
  We should point out that the conclusion in Theorem \ref{thm} still holds true if $D$ is a steady patch, that is, $\Omega=0$. This can be deduced from the following radial symmetry property of semilinear elliptic equations with monotone nonlinearity. Fraenkel (\cite{Frae}, Corollary 3.9) proved that if $\psi$ is the solution to the following semilinear elliptic equation
  \begin{equation}\label{semi}
   \begin{cases}
     -\Delta\psi=f(\psi), &\mbox{in } \mathbb D_1, \\
     \psi>0,&\mbox{in } \mathbb D_1,\\
     \psi=0, & \mbox{on }\partial\mathbb D_1,
   \end{cases}
  \end{equation}
 where $f$ has a decomposition $f=f_1+f_2$ such that $f_1:[0,+\infty)\to\mathbb R$ is locally Lipschitz continuous and $f_2:[0,+\infty)\to\mathbb R$ is nondecreasing and $f_2\equiv0$ on $[0,\kappa]$ for some $\kappa>0$, then $\psi$ must be a radial function. The proof is based on the moving plane method. From Fraenkel's result, we can easily get radial symmetry for simply-connected patches if $\Omega=0$, since in this situation the stream function $\psi:=\mathcal G I_D$ satisfies
   \begin{equation}\label{psi}
   \begin{cases}
     -\Delta\psi=I_{\{\psi>\mu\}}, &\mbox{in } \mathbb D_1, \\
     \psi=0, & \mbox{on }\partial\mathbb D_1
   \end{cases}
  \end{equation}
for some $\mu>0$. The above argument has been used by Hmidi in \cite{Hm} to prove radial symmetry for simply-connected rotating patches  if $\Omega<0$ for the whole plane case, but with some additional convexity assumption. Although radial symmetry for steady patches indeed holds true,
our method in this paper are not able to deal with this simple case, since the boundary of the disc causes some inevitable trouble as we will see in the proof.
\end{remarks}

As mentioned in Section 1, $D$ is a rotating patch in $\mathbb D_r$ with vorticity strength $\lambda$ and angular velocity $\Omega$ if and only if $D/r:=\{x\in\mathbb R^2\mid rx\in D\}$ is a rotating patch in $\mathbb D_1$ with vorticity strength $1$ and angular velocity $\Omega/\lambda$. Therefore we can easily deduce the following radial symmetry property for rotating patches in $\mathbb D_r.$

\begin{corollary}\label{cor1}
Let $D$ be a $C^1$ simply-connected rotating vortex patch in $\mathbb D_r$ with vorticity strength $\lambda$ and angular velocity $\Omega$. If $\Omega\geq\max\{\frac{\lambda}{2},\frac{2\lambda l^2r^2}{(r^2-l^2)^2}\}$ or $\Omega\leq-\frac{2\lambda l^2r^2}{(r^2-l^2)^2}$, where $l=\sup_{x\in D}|x|$, then $D$ must be a disc centered at the origin.

\end{corollary}

\begin{remark}
  By letting $r$ tend to infinity, we can easily get the result proved by G\'omez-Serrano--Park--Shi--Yao in \cite{GPSY}.
\end{remark}

If $D$ is a $C^1$ rotating patch in $\mathbb D_1$ with vorticity strength $1$ and angular velocity $\Omega$, then $D^c:=\{x\in\mathbb D_1\mid x\notin D\}$ is a rotating patch with angular velocity ${1}/{2}-\Omega$. In fact, we need only notice that $\mathcal{G}I_D+\frac{\Omega}{2}|x|^2=$constant on $\partial D$ implies $\mathcal{G}I_{D^c}+(\frac{1}{4}-\frac{\Omega}{2})|x|^2=$constant on $\partial D^c$. Taking into account Theorem \ref{thm}, we can easily get the following corollary.
\begin{corollary}\label{cor}
Let $D\subset\mathbb D_1$ be the complement of a $C^1$ simply-connected domain.
If $D$ is a rotating vortex patch with vorticity strength 1 and angular velocity ${\Omega}$ with $\Omega\leq\max\{0,\frac{1}{2}-\frac{2l^2}{(1-l^2)^2}\}$ or $\Omega\geq\frac{1}{2}+\frac{2l^2}{(1-l^2)^2}$, $l=\sup_{x\in D^c}|x|$, then $D$ must be an annulus.
\end{corollary}

\section{Proof}
In this section, we give the proof of Theorem \ref{thm}. To begin with, let us explain the basic idea of the proof.
As mentioned in Section 1, we need to deform $D$ under a suitable area-preserving flow, and show that if $D$ is not a disc centered at the origin, then the variation of the energy function $E_\Omega$(defined by \eqref{E}) is strictly positive or strictly negative if $|\Omega|$ is very large, which will lead to a contradiction. More specifically, for any divergence-free vector field $\mathbf{w},$ we define the area-preserving flow $\Phi_s$ by solving the following ordinary differential equation
\begin{equation}\label{ORD}
\begin{cases}\frac{d\Phi_s(x)}{ds}=\mathbf{w}(\Phi_s(x)), &s\in\mathbb R, \\
\Phi_0(x)=x.
\end{cases}
\end{equation}
We consider the variation of $E_\Omega$ along the flow $\Phi_s$, that is, $\frac{dE_\Omega(D_s)}{ds}\big|_{s=0}$ with \[D_s=\Phi_s(D):=\{\Phi_s(x)\mid x\in D\}.\]
By direct calculation, we have
\[\frac{dE_\Omega(D_s)}{ds}\bigg|_{s=0}=\int_D\nabla(\mathcal{G}I_D(x)+\frac{\Omega}{2}|x|^2)\cdot\mathbf{w}dx.\]
For sufficiently smooth $\mathbf{w}$ defined in $\mathbb D_1$, we can easily check $\frac{dE_\Omega(D_s)}{ds}\big|_{s=0}=0$ by integration by part. In the following lemma, we show that $\mathbf{w}\in H^1(D;\mathbb R^2)$ is sufficient.

\begin{lemma}\label{lem}
Let $\mathbf{w}\in H^1(D;\mathbb R^2)$ be a divergence-free vector field in $D$. Then
\begin{equation}\label{w}\int_D\nabla(\mathcal{G}I_D(x)+\frac{\Omega}{2}|x|^2)\cdot\mathbf{w}(x)dx=0.\end{equation}
\end{lemma}

\begin{proof}
Since $D$ is a rotating patch with angular velocity $\Omega$, we have $\mathcal{G}I_D+\frac{\Omega}{2}|x|^2=\mu$ on $\partial D$ for some $\mu\in \mathbb R$.
Moreover, by elliptic regularity theory we have $\mathcal{G}I_D+\frac{\Omega}{2}|x|^2\in C^1(\overline{D};\mathbb R^2)$. To verify \eqref{w}, we firstly assume $\mathbf{w}\in C^1(\overline{D};\mathbb R^2)$. By integration by parts we have
\begin{equation}\label{den}
\begin{split}
&\int_D\mathbf{w}(x)\cdot\nabla(\mathcal{G}I_D(x)+\frac{\Omega}{2}|x|^2)dx\\
=&\int_D\nabla\cdot(\mathbf{w}(x)(\mathcal{G}I_D(x)+\frac{\Omega}{2}|x|^2))dx-\int_D\nabla\cdot\mathbf{w}(x)(\mathcal{G}I_D(x)+\frac{\Omega}{2}|x|^2)dx\\
=&\int_{\partial D}(\mathcal{G}I_D(x)+\frac{\Omega}{2}|x|^2)\mathbf{w}(x)\cdot\vec{\nu}(x)d\mathcal{H}^1-\int_D\nabla\cdot\mathbf{w}(x)(\mathcal{G}I_D(x)+\frac{\Omega}{2}|x|^2)dx\\
=&\mu\int_{\partial D}\mathbf{w}(x)\cdot\vec{\nu}(x)d\mathcal{H}^1-\int_D\nabla\cdot\mathbf{w}(x)(\mathcal{G}I_D(x)+\frac{\Omega}{2}|x|^2)dx\\
=&\mu\int_{D}\nabla\cdot\mathbf{w}(x)dx-\int_D\nabla\cdot\mathbf{w}(x)(\mathcal{G}I_D(x)+\frac{\Omega}{2}|x|^2)dx,\\
\end{split}
\end{equation}
where $\vec{\nu}$ is the outward unit normal of $\partial D$. Since $C^1(\overline{D};\mathbb R^2)$ is dense in $H^1(D;\mathbb R^2)$, so by density argument \eqref{den} in fact holds for any $\mathbf{w}\in H^1(D;\mathbb R^2)$. Therefore the lemma is proved by the fact that $\mathbf{w}$ is a divergence-free field.
\end{proof}

To continue, we need to choose a suitable $\mathbf w$. Here we follow the choice in \cite{GPSY}, that is, $\mathbf w=x+\nabla p,$ where $p\in H^1_0(D)$ is the solution of the following elliptic problem
\begin{equation}\label{p}
\begin{cases}
  -\Delta p=2, & x\in D, \\
  x=0, & x\in\partial D.
\end{cases}
\end{equation}
In the following we begin to calculate the quantity $\int_D\nabla(\mathcal{G}I_D(x)+\frac{\Omega}{2}|x|^2)\cdot\mathbf{w}(x)dx$. For simplicity, we denote
\[\mathcal{I}_\Omega:=\int_D\nabla(\mathcal{G}I_D(x)+\frac{\Omega}{2}|x|^2)\cdot(x+\nabla p(x))dx,\]
\[\psi_\Omega(x):=\mathcal{G}I_D+\frac{\Omega}{2}|x|^2.\]
We will show that if $D$ is not a disc centered at the origin and $|\Omega|$ is sufficiently large, then $|\mathcal I{}_\Omega|>0$, which leads to a contradiction.

By direct calculation, we have
\begin{equation}\label{calculation}
\begin{split}
\mathcal{I}_\Omega&=\int_D(x+\nabla p)\cdot\nabla \psi_\Omega dx\\
&=\int_Dx\cdot\nabla_x\int_DG(x,y)dydx+\Omega\int_D|x|^2dx+\int_D\nabla p(x)\cdot\nabla \psi_\Omega(x) dx\\
&=\int_Dx\cdot\nabla_x\int_D-\frac{1}{2\pi}\ln|x-y|-h(x,y)dydx+\Omega\int_D|x|^2dx-\int_Dp(x)\Delta \psi_\Omega(x) dx\\
&=-\frac{1}{2\pi}\int_D\int_D\frac{x\cdot(x-y)}{|x-y|^2}dxdy-\int_Dx\cdot\nabla_x\int_Dh(x,y)dydx+\Omega\int_D|x|^2dx-\int_Dp(x)\Delta \psi_\Omega(x) dx\\
&=-\frac{1}{4\pi}\int_D\int_D\frac{|x-y|^2}{|x-y|^2}dxdy-\int_Dx\cdot\nabla_x\int_Dh(x,y)dydx+\Omega\int_D|x|^2dx-\int_Dp(x)\Delta \psi_\Omega(x) dx\\
&=-\frac{1}{4\pi}|D|^2-\int_Dx\cdot\nabla_x\int_Dh(x,y)dydx+\Omega\int_D|x|^2dx+(1-2\Omega)\int_Dp(x)dx\\
&=(2\Omega-1)\left(\frac{1}{4\pi}|D|^2-\int_Dp(x)dx\right)+\Omega\left(\int_D|x|^2dx-\frac{1}{2\pi}|D|^2\right)-\int_Dx\cdot\nabla_x\int_Dh(x,y)dydx,
\end{split}
\end{equation}
where we used the antisymmetry of the function $\frac{x\cdot(x-y)}{|x-y|^2}$ and the fact $-\Delta\psi_\Omega=1-2\Omega$ in $D$.
Therefore we need to deal the following three terms in \eqref{calculation}
\[\mathcal{J}_\Omega:=(2\Omega-1)\left(\frac{1}{4\pi}|D|^2-\int_Dp(x)dx\right)\]
\[\mathcal{K}_\Omega:=\Omega\left(\int_D|x|^2dx-\frac{1}{2\pi}|D|^2\right)\]
\[\mathcal{L}:=-\int_Dx\cdot\nabla_x\int_Dh(x,y)dydx.\]

To deal with $\mathcal J_\Omega$, we need the following result proved by Talenti in \cite{Talenti}.
\begin{proposition}[\cite{Talenti}, Theorem 1]\label{talenti}
  Let $D$ be a planar bounded domain with $C^1$ boundary, and let $p$ be the solution of the following equation
  \begin{equation}\label{p2}
\begin{cases}
  -\Delta p=2, & x\in D, \\
  x=0, & x\in\partial D.
\end{cases}
\end{equation}
Then we have
\[\int_Dp(x)dx\leq\frac{1}{4\pi}|D|^2.\]
Moreover, the equality is achieved if and only if $D$ is a disc.
\end{proposition}

To deal with $\mathcal{K}_\Omega,$ we need the following lemma.
\begin{lemma}\label{K}
Let $D$ be a measurable set in $\mathbb R^2$. Then we have
\[\int_D|x|^2dx\geq\frac{1}{2\pi}|D|^2+\frac{1}{\pi}|D\setminus B|^2,\]
where $B$ is the disc centered at the origin with the same area as $D$. Moreover, the equality is achieved if and only if $D$ is an annulus centered at the origin.
\end{lemma}
\begin{proof}
   Notice that
   \[\int_D|x|^2dx=\int_B|x|^2dx+\int_{D\setminus B}|x|^2dx-\int_{B\setminus D}|x|^2dx=\frac{1}{2\pi}|D|^2+\int_{D\setminus B}|x|^2dx-\int_{B\setminus D}|x|^2dx.\]
   So it suffices to prove
   \[\int_{D\setminus B}|x|^2dx-\int_{B\setminus D}|x|^2dx\geq\frac{1}{\pi}|D\setminus B|^2.\]
To this end, we notice that for $|D\setminus B|$ being fixed, $\int_{D\setminus B}|x|^2dx$ attains its minimum value if and only if $D\setminus B$ is an annulus with its inner boundary coinciding with $\partial B$. Similarly, for $|B\setminus D|$ being fixed, $\int_{B\setminus D}|x|^2dx$ attains its maximum value if and only if $B\setminus D$ is an annulus contained in $B$ with its outer boundary coinciding with $\partial B$. Therefore by direct calculation we can easily get
\[\int_{D\setminus B}|x|^2dx\geq \frac{2|B||D\setminus B|+|D\setminus B|^2}{2\pi},\]
\[\int_{B\setminus D}|x|^2dx\leq \frac{2|B||D\setminus B|-|D\setminus B|^2}{2\pi}.\]
Now we can subtracting the above two inequalities to get the desired result.

\end{proof}

By Proposition \ref{talenti} and Lemma \ref{K} we see that if $\Omega\geq{1}/{2}$, then $\mathcal I_\Omega+\mathcal J_\Omega\geq \frac{\Omega}{\pi}|D\setminus B|^2$, and if $\Omega\leq0$ and $|\Omega|$ is sufficiently large, then $\mathcal I_\Omega+\mathcal J_\Omega\leq \frac{\Omega}{\pi}|D\setminus B|^2$. If $|D\setminus B|>0$, in order to get a contradiction, we need to show that $|\mathcal L|$ is controlled by a reasonable constant multiple of $|D\setminus B|^2$. This is exactly what we will do in the following lemma.

\begin{lemma}\label{L}
For any measurable set $D\subset \mathbb D_1,$ define $l=sup_{x\in D}|x|$. Then the following inequality holds
\[\left|\mathcal{L}\right|\leq \frac{2l^2}{\pi(1-l^2)^2}|D\setminus B|^2.\]
\end{lemma}
\begin{proof}
Recall
\[h(x,y)=-\frac{1}{2\pi}\ln\left|\frac{x}{|x|}-{|x|y}\right|.\]
By direct calculation, we have
\begin{equation*}
\begin{split}
\mathcal{L}=-\int_Dx\cdot\nabla_x\int_Dh(x,y)dydx&=\frac{1}{2\pi}\int_Dx\cdot\nabla_x\int_D\ln\left|\frac{x}{|x|}-|x|y\right|dydx\\
&=\frac{1}{4\pi}\int_Dx\cdot\nabla_x\int_D\ln(1-2x\cdot y+|x|^2|y|^2)dydx\\
&=\frac{1}{2\pi}\int_D\int_D\frac{|x|^2|y|^2-x\cdot y}{1-2x\cdot y+|x|^2|y|^2}dxdy
\end{split}
\end{equation*}
So it suffices to prove
\begin{equation}\label{divi}
\left|\frac{1}{2\pi}\int_D\int_D\frac{|x|^2|y|^2-x\cdot y}{1-2x\cdot y+|x|^2|y|^2}dxdy\right|\leq\frac{2l^2}{\pi(1-l^2)^2}|D\setminus B|^2.
\end{equation}
For simplicity we denote \[L(x,y):=\frac{1}{2\pi}\frac{|x|^2|y|^2-x\cdot y}{1-2x\cdot y+|x|^2|y|^2}.\]
We divide the integral in \eqref{divi} into the following nine parts
\begin{equation}\label{nine}
\begin{split}
\int_D\int_DL(x,y)dxdy&=\int_B\int_BL(x,y)dxdy-\int_B\int_{B\setminus D}L(x,y)dxdy+\int_B\int_{D\setminus B}L(x,y)dxdy\\
&-\int_{B\setminus D}\int_BL(x,y)dxdy+\int_{B\setminus D}\int_{B\setminus D}L(x,y)dxdy-\int_{B\setminus D}\int_{D\setminus B}L(x,y)dxdy\\
&+\int_{D\setminus B}\int_BL(x,y)dxdy-\int_{D\setminus B}\int_{B\setminus D}L(x,y)dxdy+\int_{D\setminus B}\int_{D\setminus B}L(x,y)dxdy\\
&=\int_B\int_BL(x,y)dxdy-2\int_B\int_{B\setminus D}L(x,y)dxdy+2\int_B\int_{D\setminus B}L(x,y)dxdy\\
&+\int_{B\setminus D}\int_{B\setminus D}L(x,y)dxdy+\int_{D\setminus B}\int_{D\setminus B}L(x,y)dxdy-2\int_{D\setminus B}\int_{B\setminus D}L(x,y)dxdy,
\end{split}
\end{equation}
where we used the symmetry of the function $L$.

Now we claim
\begin{equation}\label{claim}
\int_B\int_BL(x,y)dxdy=\int_B\int_{B\setminus D}L(x,y)dxdy=\int_B\int_{D\setminus B}L(x,y)dxdy=0.
\end{equation}
To prove \eqref{claim}, we prove a more general result, that is, for any measurable set $E$ there holds
\begin{equation}\label{clai}
\int_B\int_EL(x,y)dxdy=\int_E\int_{B}L(x,y)dxdy=0.
\end{equation}
In fact, by using the polar coordinates we have
\[\int_E\int_{B}L(x,y)dxdy=\int_E\int_0^b\int_0^{2\pi}\frac{1}{2\pi}\frac{\rho^2|y|^2-\rho|y|\cos(\theta)}{1-2\rho|y|\cos(\theta)
+\rho^2|y|^2}d\theta d\rho dy,\]
where $b$ is the radius of $B$. Notice that for fixed $\rho,y$, we have $\rho|y|\in(0,1)$. We can use Cauchy's residue theorem to calculate the above integral to obtain
\[\int_0^{2\pi}\frac{1}{2\pi}\frac{\rho|y|\cos(\theta)-\rho^2|y|^2}{1-2\rho|y|\cos(\theta)
+\rho^2|y|^2}d\theta =0\]
for fixed $y$ and $\rho$, which proves \eqref{clai}.

We continue to estimate $\mathcal{L}$. From the above claim we deduce
\[|\mathcal{L}|=\left|\int_{B\setminus D}\int_{B\setminus D}L(x,y)dxdy+\int_{D\setminus B}\int_{D\setminus B}L(x,y)dxdy-2\int_{D\setminus B}\int_{B\setminus D}L(x,y)dxdy\right|.\]
By simple calculation, we can obtain the following
upper bound and lower bound for $L(x,y)$
\[L(x,y)\geq -\frac{l^2}{2\pi(1-l^2)},\,\,\forall\,x,y\in B\cup D,\]
\[L(x,y)\leq\frac{l^2(1+l^2)}{2\pi(1-l^2)^2},\,\,\forall\,x,y\in B\cup D,\]
where $l=\sup_{x\in D}|x|$.
Thus we obtain
\[|\mathcal{L}|\leq \frac{2l^2}{\pi(1-l^2)^2}|D\setminus B|^2,\]
which completes the proof.

\end{proof}

\begin{proof}[Proof of Theorem \ref{thm}]
Assume that $D$ is not a disc centered at the origin, or equivalently $|D\setminus B|>0$. We will deduce a contradiction if $\Omega\geq\max\{\frac{1}{2},\frac{2 l^2}{(1-l^2)^2}\}$ or $\Omega\leq-\frac{2 l^2}{(1-l^2)^2}$.

First we consider the case $\Omega\geq\max\{\frac{1}{2},\frac{2 l^2}{(1-l^2)^2}\}$. By Proposition \ref{talenti}, Lemma \ref{K} and Lemma \ref{L}, we obtain
\[0=\mathcal I_\Omega=\mathcal J_\Omega+\mathcal K_\Omega+\mathcal L\geq \frac{\Omega}{\pi}|D\setminus B|^2-|\mathcal{L}|\geq \frac{\Omega}{\pi}|D\setminus B|^2- \frac{2l^2}{\pi(1-l^2)^2}|D\setminus B|^2.\]
If $\Omega>\frac{2 l^2}{(1-l^2)^2}$, then we get an obvious contradiction. If $\Omega=\frac{2 l^2}{(1-l^2)^2},$ then $D$ must be an annulus centered at the origin. In fact, if $D$ is not an annulus centered at the origin, then the inequality in Lemma \ref{K} is strict, which leads to the following contradiction
\[0=\mathcal I_\Omega=\mathcal J_\Omega+\mathcal K_\Omega+\mathcal L> \frac{\Omega}{\pi}|D\setminus B|^2-|\mathcal{L}|\geq \frac{\Omega}{\pi}|D\setminus B|^2- \frac{2l^2}{\pi(1-l^2)^2}|D\setminus B|^2= 0.\]
But once $D$ is an annulus centered at the origin, we can repeat the calculation of \eqref{clai} to obtain $\mathcal{L}=0$,
thus
\[0=\mathcal I_\Omega=\mathcal J_\Omega+\mathcal K_\Omega+\mathcal L> \frac{\Omega}{\pi}|D\setminus B|^2-|\mathcal{L}|=\frac{\Omega}{\pi}|D\setminus B|^2>0,\]
which is also a contradiction.

Now we consider the case $\Omega\leq-\frac{2 l^2}{(1-l^2)^2}$.  By Proposition \ref{talenti}, Lemma \ref{K} and Lemma \ref{L}, we obtain
\[0=\mathcal I_\Omega=\mathcal J_\Omega+\mathcal K_\Omega+\mathcal L\leq\frac{\Omega}{\pi}|D\setminus B|^2+|\mathcal L|\leq \frac{\Omega}{\pi}|D\setminus B|^2+ \frac{2l^2}{\pi(1-l^2)^2}|D\setminus B|^2.\]
If $\Omega<-\frac{2l^2}{(1-l^2)^2}$, then we get an obvious contradiction. If $\Omega=-\frac{2l^2}{(1-l^2)^2}$, we can still deduce that $D$ is an annulus centered at the origin. If otherwise, by Lemma \ref{K} we have
\[0=\mathcal I_\Omega=\mathcal J_\Omega+\mathcal K_\Omega+\mathcal L<\frac{\Omega}{\pi}|D\setminus B|^2+|\mathcal L|\leq \frac{\Omega}{\pi}|D\setminus B|^2+ \frac{2l^2}{\pi(1-l^2)^2}|D\setminus B|^2=0,\]
which is a contradiction. But once $D$ is an annulus centered at the origin, we immediately know $\mathcal{L}=0,$ therefore
\[0=\mathcal I_\Omega=\mathcal J_\Omega+\mathcal K_\Omega+\mathcal L\leq\frac{\Omega}{\pi}|D\setminus B|^2+\mathcal L= \frac{\Omega}{\pi}|D\setminus B|^2<0,\]
which also leads to a contradiction.
\end{proof}

From the proof, we see that the bound $\Omega\geq\max\{\frac{1}{2},\frac{2 l^2}{(1-l^2)^2}\}$ or $\Omega\leq-\frac{2 l^2}{(1-l^2)^2}$ is not optimal. In fact, one can improve this bound by giving $\mathcal L$ a more accurate estimate. But for a general domain $D$ this is usually very difficult.

\noindent{\bf Acknowledgments:}
 We would like to thank Piye Sun for the useful discussion about the estimate of $\mathcal L.$

\renewcommand\refname{References}
\renewenvironment{thebibliography}[1]{%
\section*{\refname}
\list{{\arabic{enumi}}}{\def\makelabel##1{\hss{##1}}\topsep=0mm
\parsep=0mm
\partopsep=0mm\itemsep=0mm
\labelsep=1ex\itemindent=0mm
\settowidth\labelwidth{\small[#1]}%
\leftmargin\labelwidth \advance\leftmargin\labelsep
\advance\leftmargin -\itemindent
\usecounter{enumi}}\small
\def\newblock{\ }
\sloppy\clubpenalty4000\widowpenalty4000
\sfcode`\.=1000\relax}{\endlist}
\bibliographystyle{model1b-num-names}
%\bibliography{refs}

\end{document}